\documentclass{amsart}

	\usepackage[utf8]{inputenc}

	\usepackage{amsmath,amsfonts,amssymb,amsthm}

	\usepackage[T1]{fontenc}

	\usepackage{etoolbox}
	\patchcmd{\section}{\scshape}{\scshape\bfseries}{}{}
	\makeatletter
	\renewcommand{\@secnumfont}{\scshape\bfseries}
	\makeatother

	\let\epsilon\varepsilon

	\usepackage[indent=10pt]{parskip}

	\usepackage[bookmarks=true,bookmarksopen=true]{hyperref}

	\numberwithin{equation}{section}
    
    \newtheorem{theorem}[equation]{Theorem}

	\newtheorem{lemma}[equation]{Lemma}
	\newtheorem{proposition}[equation]{Proposition}
		
	\theoremstyle{plain}
	\newtheorem{result}{Theorem}
	
	\newtheorem{corollaryresult}[result]{Corollary}

	\theoremstyle{definition}

	\theoremstyle{remark}
	\newtheorem{remark}[equation]{Remark}
	
	\numberwithin{table}{section}

	\usepackage{tikz}
	\usetikzlibrary{matrix}

	\usepackage{todonotes}

	\usepackage{enumerate}

	\usepackage{float}

	\makeatletter
	\def\blfootnote{\gdef\@thefnmark{}\@footnotetext}
	\makeatother

	\DeclareMathOperator{\diag}{diag}

	\DeclareMathOperator{\Ric}{Ric}
        \DeclareMathOperator{\Scal}{Scal}

	\DeclareMathOperator{\trace}{trace}

	\newcommand{\of}[1]{\left( #1 \right)}

	\newcommand{\gG}{\mathsf{G}}
	
	\newcommand{\gH}{\mathsf{H}}

	\newcommand{\SO}{\mathsf{SO}}

	\newcommand{\Spin}{\mathsf{Spin}}

	\newcommand{\fg}{\mathfrak{g}}

	\newcommand{\fm}{\mathfrak{m}}
	
	\newcommand{\fp}{\mathfrak{p}}
	
	\newcommand{\so}{\mathfrak{so}}

	\newcommand{\bO}{\mathbb{O}}

	\newcommand{\bR}{\mathbb{R}}

	\mathcode`l="8000
	\begingroup
	\makeatletter
	\lccode`\~=`\l
	\DeclareMathSymbol{\lsb@l}{\mathalpha}{letters}{`l}
	\lowercase{\gdef~{\ifnum\the\mathgroup=\m@ne \ell \else \lsb@l \fi}}%
	\endgroup

\usepackage[foot]{amsaddr}

\title[$\gG/\gH$ with two isotropy summands develops $\Ric>0$ under Ricci flow]{Homogeneous spaces with two equivalent isotropy summands develop positive Ricci curvature under Ricci flow}

\author[E. Cochran]{Eric Cochran$^1$}
\address{Department of Mathematics, University of California, Irvine, USA$^1$}
\email{ecochra2@uci.edu}

\author[A. Mingajev]{Arseny Mingajev$^{2}$}
\address{Department of Mathematics, Trinity University, San Antonio, TX, USA$^2$}
\email{amingaje@trinity.edu}

\author[L. Mouillé]{Lawrence Mouillé$^2$}
\email{lmouille@trinity.edu}

\author[N. Valiyakath]{Nazia Valiyakath$^3$}
\address{Department of Mathematics, University of Rochester, Rochester, NY, USA$^3$}
\email{naziavaliyakath@gmail.com}

\subjclass[2020]{53E20, 53C20, 53C21, 53C30}

\date{\today}

\begin{document}

\begin{abstract}
    We study normalized Ricci flow on simply connected homogeneous spaces $\gG/\gH$ for which the isotropy representation splits into exactly two equivalent irreducible subrepresentations.
    We prove that every $\gG$-invariant metric evolves to one with positive Ricci curvature, and that the family of $\gG$-invariant metrics with positive Ricci curvature is forward-invariant under the flow. 
    The proof relies on the fact that the phase portrait of the family of fixed-volume $\gG$-invariant metrics can be explicitly visualized. 
\end{abstract}

\maketitle

\section{Introduction}


Ricci flow, introduced by Hamilton in \cite{Hamilton82}, constitutes a powerful geometric evolution equation for deforming Riemannian metrics. 
Given a Riemannian manifold $(M,g_0)$, the flow is defined by the partial differential equation
\[
    \frac{\partial}{\partial t} g(t) = -2 \Ric_{g(t)}, \qquad g(0) = g_0.
\]
Heuristically, much like the heat equation, Ricci flow is expected to have a regularizing effect, evolving a given Riemannian metric to some \textit{canonical} metric on $M$.  
A general goal is to then draw topological or geometric conclusions from the existence of such a metric.
When combined with Perelman’s monotonicity functionals and entropy techniques, Ricci flow has been instrumental in resolving several long-standing conjectures in geometry and topology, most notably the Poincaré and Geometrization Conjectures.

Ricci flow is known to preserve several important curvature conditions, including nonnegative scalar curvature \cite{Hamilton82}, positive-semidefinite curvature operator \cite{Hamilton86}, and nonnegative isotropic curvature \cite{BrendleSchoen09,Hamilton97,Nguyen10}.
Brendle and Schoen subsequently used this last property in their ground-breaking proof of the Differentiable Sphere Theorem in \cite{BrendleSchoen09}.
A unified approach to proving invariance of nonnegative curvature conditions was developed by Wilking in \cite{Wilking13}
However, not all lower curvature bounds are preserved under the flow; see the work of Ni \cite{Ni04}, B\"ohm and Wilking \cite{BohmWilking07}, M\'aximo \cite{Maximo11,Maximo14}, and Bettiol and Krishnan \cite{BettiolKrishnan19,BettiolKrishnan23}.

An important property of Ricci flow is that it preserves the symmetries of the initial metric $g_0$.
This property follows from the fact that the Ricci tensor is equivariant (or natural) with respect to pullbacks by diffeomorphisms of the manifold. 
Consequently, on homogeneous spaces $\gG/\gH$, Ricci flow reduces to a system of ordinary differential equations when restricted to the the space of $\gG$-invariant metrics. 
In cases that the isotropy representation decomposes into few irreducible subrepresentations (isotropy summands), the space of $\gG$-invariant metrics is low-dimensional.
Furthermore, by considering normalized Ricci flow given by
\[
    \frac{\partial}{\partial t} g(t) = -2 \Ric_{g(t)} + \frac{2 \Scal}{\dim (M)} g(t), \qquad g(0)=g_0,
\]
the family of $\gG$-invariant metrics with normalized volume is preserved under the flow.
Restricting to such a family of metrics produces an even lower-dimensional system to analyze, and the dynamics of normalized Ricci flow on this entire space of metrics can be completely visualized.

This reduction to low-dimensional ODE systems has been exploited by several authors to study the preservation of curvature conditions.
Abiev and Nikonorov investigated the preservation of positive sectional and Ricci curvature on Wallach spaces in \cite{AbievNikonorov16}, 
Abiev studied Stiefel manifolds $\SO(n)/\SO(n-2)$ in \cite{Abiev25b}, and DeVito, Gonz\'alez-\'Alvaro, and Zarei focused on spheres and complex projective spaces in \cite{DGAZpreprint}.
For more results on homogeneous Ricci flow, see the work of Abiev \cite{Abiev25a}, Abiev et.~al.~\cite{AANS14}, B\"ohm \cite{Bohm15}, B\"ohm and Lafuente \cite{BohmLafuente18}, Buzano \cite{Buzano14}, Cavenaghi et.~al.~\cite{CGM}, Gonz\'alez-\'Alvaro, and Zarei \cite{GAZ24,GAZ25}, Lauret \cite{Lauret13}, Sbiti \cite{Sbiti20}, and Statha \cite{Statha22}.

In this paper, we investigate the preservation of positive Ricci curvature for $\gG$-invariant Riemannian metrics under normalized Ricci flow on simply connected homogeneous spaces $\gG/\gH$ where $\gG$ is compact and semisimple, $\gH$ is closed and connected, and the isotropy representation splits into \textit{exactly two equivalent} irreducible summands. 
In the case that $\gG$ is simple, it follows from the classification by Dickinson and Kerr in \cite{DickinsonKerr08} that $\gG/\gH$ is $\Spin(8)/\gG_2$, which is diffeomorphic to $S^7 \times S^7$.
Our main result is the following:

\begin{result}\label{mainthm:Ricpos}
        Every $\Spin(8)$-invariant metric on $\Spin(8)/\gG_2$ evolves to a metric with positive Ricci curvature under normalized Ricci flow.
        Furthermore, the family of $\Spin(8)$-invariant metrics with positive Ricci curvature is forward-invariant under the flow.
\end{result}

Alternatively, in the case that $\gG/\gH$ has exactly two irreducible isotropy summands and $\gG$ is \textit{not} simple, Pulemotov and Ziller observed in \cite[Section 5.2]{PulemotovZiller21preprint} that $\gG/\gH$ must be a Ledger-Obata space $(\gH\times \gH\times \gH)/\Delta \gH$, where $\gH$ is simple and $\Delta \gH$ denotes the diagonal embedding of $\gH$ into $\gH^3$. 
They also show that the scalar curvature functional of $\gH^3$-invariant metrics on all Ledger-Obata spaces are proportional to that of $\Spin(8)/\gG_2$.
Since normalized Ricci flow is the $L^2$-gradient flow of the scalar curvature, it follows that the ODE systems of normalized Ricci flow for these spaces all coincide. 
Consequently, Theorem \ref*{mainthm:Ricpos} immediately yields the following:

\begin{corollaryresult}\label{cor:LedgerObata}
    Let $\gG/\gH$ be a simply connected homogeneous space for which $\gG$ is compact and semisimple, $\gH$ is closed and connected, and the isotropy representation splits into exactly two equivalent irreducible summands.
    Then every $\gG$-invariant metric on $\gG/\gH$ evolves to a metric with positive Ricci curvature under normalized Ricci flow.
    Furthermore, the family of $\gG$-invariant metrics with positive Ricci curvature is forward-invariant under the flow.
\end{corollaryresult}

We note that our result has a conclusion similar to Abiev's main result on the Stiefel manifolds $\SO(n)/\SO(n-2)$ in \cite{Abiev25b}.
While Abiev exploits extra symmetries to diagonalize the $\SO(n)$-invariant metrics on $\SO(n)/\SO(n-2)$ (see \cite[Section 4]{Kerr98}), the family of $\gG$-invariant metrics we consider are not simultaneously diagonalizable.
Consequently, we must analyze the full family of $\gG$-invariant metrics rather than a diagonal subfamily.

Our proof relies on the fact that the family of $\gG$-invariant metrics with fixed volume is two-dimensional in our case.
Thus, the dynamics of the ODE system associated with normalized Ricci flow can be completely visualized.
The spaces we consider here are among the few remaining for which the family of $\gG$-invariant metrics is so low-dimensional while the behavior of normalized Ricci flow has not been studied.
    
\subsection{Organization of paper}

In Section \ref{sec:invariantmetrics}, we give an overview of the isotropy representation, family of $\Spin(8)$-invariant metrics, Ricci curvature, and scalar curvature of $\Spin(8)/\gG_2$.
We then write the normalized Ricci flow equation and highlight the fixed points and several important invariant curves for the system.
In Section \ref{sec:mainproof}, we adopt a more convenient coordinate system for analyzing the system and prove Theorem \ref{mainthm:Ricpos}.
We point out that in several arguments, we rely on the IsEmpty function from Maple’s RegularChains package and SemiAlgebraicSetTools subpackage to prove that certain sets of polynomial inequalities imply others.
The Maple code we used to justify these claims is available on the GitHub repository \cite{GitHub}.
For an overview of this subpackage, see \cite{MapleSite} and the references therein.

\subsection{Acknowledgments}

We would like to thank Jason DeVito, David Gonz\'alez-\'Alvaro, Lee Kennard, Megan Kerr, Yuri\u{i} Nikonorov, Artem Pulemotov, William Wylie, Masoumeh Zarei, and Wolfgang Ziller for helpful discussions. The first author was partially supported by NSF grant DMS-2342135. The second and third authors were supported by Semmes Distinguished Scholars in Science scholarships, and the third author was supported by an AMS-Simons Research Enhancement Grants for PUI Faculty.

    



\section{Ricci flow equation on the family of invariant metrics}
\label{sec:invariantmetrics}

\subsection{Invariant metrics}
    By the classification carried out by Dickinson and Kerr in \cite{DickinsonKerr08}, $\Spin(8)/\gG_2$ is the only simply connected homogeneous space $\gG/\gH$ such that $\gG$ is simple and compact, $\gH$ is a connected and closed subgroup, and the isotropy representation splits into exactly two equivalent subrepresentations. 

    To describe the isotropy representation of this space, we will follow \cite{Kerr98,Kerin11}.
    Letting $\bO$ denote the octonions, or Cayley numbers, there is a canonical matrix group representation for $\Spin(8)$: 
    \[
        \Spin(8) = \{(A,B,C) \in \SO(8) : A(x) B(y) = C(xy) \; \text{ for all } x, y \in \bO \}.
    \]
    $\Spin(8)$ acts transitively on the product of spheres $S^7 \times S^7 \subset \bO \times \bO$ via the map $(A,B,C) : (x,y) \mapsto (Ax,By)$, and the isotropy group of the element $(1,1)$ is 
    \[
        \gG_2 = \{ (A,B,C) \in \Spin(8) : A = B = C \}.
    \]
    Thus $\Spin(8)/\gG_2$ is diffeomorphic to $S^7 \times S^7$.

    Now $\Spin(8)$ is the universal cover of $\SO(8)$, with a double covering homomorphism $\Spin(8)\to\SO(8)$ given by $(A,B,C) \mapsto C$.
    Thus, the Lie algebra of $\Spin(8)$ can be identified with $\so(8) = \{A \in M_8(\bR) : A^t = -A\}$.
    For an appropriately ordered basis of $\bO$, we have the inclusions of Lie groups
    \[
        \gG_2 \subset 
        \begin{pmatrix}
            1 & 0 \\
            0 & \SO(7)
        \end{pmatrix}
        \subset \SO(8),
    \]
    which induce inclusions of Lie algebras $\fg_2 \subset \diag(0,\so(7)) \subset \so(8)$.
    We will identify $\fg_2$ and $\so(7)$ with their inclusions into $\so(8)$.
    Hence, $\fg_2$ consists of matrices of the form
    \[
        \left( 
            \begin{array}{c|ccc}
              0 &  & 0 &  \\ \hline
               &  &  &  \\
              0 &  & U &  \\
               &  &  &  
            \end{array} 
        \right),
    \]
    where $U$ is a skew-symmetric matrix given by
    \[
        \begin{pmatrix}
            0                & u_1+u_2 & u_7+u_8 & u_3+u_4 & u_9+u_{10} & u_5+u_6           & u_{11}+u_{12} \\
            *       & 0       & u_{13}  & -u_{11} & u_5        & -u_9              & u_3           \\
            *      & * & 0       & u_6     & u_{12}     & -u_4              & -u_{10}       \\
            *      & *  & *    & 0       & u_{14}     & u_7               & -u_1          \\
            *    & *    & * & * & 0          & u_2               & u_8           \\
            *      & *    & *    & *    & *       & 0                 & u_{13}+u_{14} \\
            * & *    & *  & *     & *       & *  & 0             \\
        \end{pmatrix},
    \]
    where $(u_1, \dots , u_{14}) \in \bR^{14}$.
    Now any bi-invariant metric on $\SO(8)$ is given as a negative real scalar multiple of the Killing form on $\so(8)$.
    We will choose $Q$ to be the bi-invariant metric on $\SO(8)$ given by
    \[
        Q(X,Y) = - \frac{1}{2} \trace(XY) \text{ for all } X,Y \in \so(8),
    \]
    We then have the $Q$-orthogonal decomposition
    \[
        \so(7) = \fm \oplus \fg_2,
    \]
    where $\fm$ consists of matrices parametrized by $\vec{v} = (v_1,\dots,v_7) \in \bR^7$ of the form
    \begin{equation}
    \label{eq:M}
        M(\vec{v}) = 
        \begin{pmatrix}
            0   & 0       & 0     & 0     & 0     & 0     & 0     & 0     \\
            0   & 0       & v_1   & v_2   & v_3   & v_4   & v_5   & v_6   \\
            0   & -v_1    & 0     & v_7   & v_6   & -v_5  & v_4   & -v_3  \\
            0   & -v_2    & -v_7  & 0     & -v_5  & -v_6  & v_3   & v_4   \\
            0   & -v_3    & -v_6  & v_5   & 0     & v_7   & -v_2  & v_1   \\
            0   & -v_4    & v_5   & v_6   & -v_7  & 0     & -v_1  & -v_2  \\
            0   & -v_5    & -v_4  & v_3   & v_2   & v_1   & 0     & -v_7  \\
            0   & -v_6    & v_3   & -v_4  & -v_1  & v_2   & v_7   & 0     \\
        \end{pmatrix}.
    \end{equation}
    Similarly, we have a $Q$-orthogonal decomposition 
    \[
        \so(8) = \fp \oplus \so(7),
    \]
    where $\fp$ is parametrized by $\vec{w} \in \bR^7$ via
    \begin{equation}
    \label{eq:P}
        P(\vec{w}) = 
        \left( 
            \begin{array}{c|ccc}
              0 &  & -\vec{w}^T &  \\ \hline
               &  &  &  \\
              \vec{w} &  & 0 &  \\
               &  &  &  
            \end{array} 
        \right)
    \end{equation}
    The isotropy representation of $\gG_2$ on $\fg_2^\perp \subset \so(8)$ splits into two equivalent subrepresentations that respect the splitting $\fg_2^\perp = \fm \oplus \fp$.
    Following \cite[Equation 21]{Kerin11}, a $\gG_2$-equivariant linear isomorphism $\psi:\fm \to \fp$ is given by
    \[
        \psi(M(v_1, \dots ,v_7)) = P(v_7, -v_2, v_1, -v_4, v_3, v_6, -v_5).
    \]
    Letting $e_1,\dots,e_7$ denote the standard basis for $\bR^7$, we fix a $Q$-orthonormal basis $\{X_1,\dots,X_{14}\}$ for $\fm \oplus \fp$ given by 
        \begin{equation}\label{eq:Xi}
            X_i = \frac{1}{\sqrt{3}}M(e_i)\text{ and }X_{i+7} = \psi(M(e_i))\text{ for }1 \leq i \leq 7.
        \end{equation}
    
    The complexifications of the $\gG_2$-representations on $\fm$ and $\fp$ are irreducible.
    Thus, it follows from Schur's lemma that with respect to the basis $\{X_i,\dots,X_{14}\}$, all $\gG_2$-invariant inner products $g$ on $\fm\oplus\fp$ are of the form 
    \[
        g = \begin{pmatrix}
            a I_7 & c I_7 \\
            c I_7 & b I_7
        \end{pmatrix},
    \]
    where $a,b,c\in\bR$, $a>0$, $b>0$, $ab-c^2 > 0$, and $I_7$ is the $7 \times 7$ identity matrix.
    The family of $\Spin(8)$-invariant metrics on $\Spin(8)/\gG_2$ is in one-to-one correspondence with the set of $\mathrm{Ad}(\gG_2)$-invariant inner products on $\fm\oplus\fp$, so we will refer to an element of either set as simply $g$.
    
        


    


\subsection{Curvature}
    Given an invariant metric $g$ on $\Spin(8)/\gG_2$, the Ricci tensor can be identified with a $G_2$-equivariant endomorphism on $\fm \oplus \fp$. 
    Thus, it also follows from Schur's lemma that the Ricci $(0,2)$-tensor for $g$ can be represented in the basis $\{X_1, \dots, X_{14}\}$ by
    \[
        \of{\Ric_{i,j}} = \begin{pmatrix}
            \alpha I_7 & \gamma I_7 \\
            \gamma I_7 & \beta I_7
        \end{pmatrix}, 
    \]
    for some real numbers $\alpha, \beta, \gamma$.
    By adapting \cite[Corollary 7.33]{Besse87}, or using \cite[Lemma 3.6]{Puttmann99}, it is straightforward to compute that 
        \begin{align*}
            \alpha &= \frac{9}{2} + \frac{(a^2 - 3 c^2)^2}{2(a b - c^2)^2},\\
            \beta &= \frac{-2 a^3 b + 3 a^2 (4 b^2 + c^2) - 24 a b c^2 + 9 b^2 c^2 + 6 c^4}{2(a b - c^2)^2}, \\
            \gamma &= \frac{a c (a - 3 b)^2}{2(a b - c^2)^2}.
        \end{align*}
        Then in the basis $\{X_1, \dots, X_{14}\}$, the Ricci $(1,1)$-tensor can be represented by
        \[
            \of{\Ric^i_j} = \of{g^{i,k} \Ric_{k,j}} = 
            \begin{pmatrix}
                \delta I_7 & \epsilon I_7 \\
                \zeta I_7 & \eta I_7
            \end{pmatrix}
        \]
        where
        \begin{equation}\label{eq:Ric11abc}
        \begin{aligned}
            \delta &= \frac{a^3 + 9 a b^2 - 18 b c^2}{2 (a b - c^2)^2}, \\
            \epsilon &= \frac{3c (a^2 - 6 a b + 3 b^2 + 2 c^2)}{2 (a b - c^2)^2}, \\
            \zeta &= -\frac{3c (a^2 - 3 c^2)}{(a b - c^2)^2},\\
            \eta &= -\frac{a (a^2 - 6 a b + 3 c^2)}{(a b - c^2)^2}.
        \end{aligned}
        \end{equation}
        
        This operator has two real eigenvalues given by the roots of the polynomial $(\delta-x)(\eta-x)-\epsilon\zeta$.
        To show that these roots are positive, it is enough to check that their sum $S=\delta+\eta$ and product $P=\delta\eta-\epsilon\zeta$ are both positive.
        Thus, we have
        \begin{align*}
            S &= -\frac{a^3 - 12 a^2 b - 9 a b^2 + 6 a c^2 + 18 b c^2}{2 (a b - c^2)^2}\\
            P &= \frac{-a^6 + 6 a^5 b - 9 a^4 b^2 + 6 a^4 c^2 + 54 a^3 b^3 - 36 a^3 b c^2 }{2 (a b - c^2)^4} \\
            &\phantom{==} + \frac{- 108 a^2 b^2 c^2 - 9 a^2 c^4 + 216 a b c^4 - 81 b^2 c^4 - 54 c^6}{2 (a b - c^2)^4}.
        \end{align*}
        
        Finally, it follows that the scalar curvature is
        \[
            \Scal = -\frac{7(a^3 - 12 a^2 b - 9 a b^2 + 6 a c^2 + 18 b c^2)}{2 (a b - c^2)^2}.
        \]

\subsection{Normalized Ricci flow}
    The normalized Ricci flow equation for an $n$-dimensional manifold $M^n$ with initial metric $g_0$ is as follows:
    \[
        \frac{\partial}{\partial t} g(t) = -2 \Ric_{g(t)} + \frac{2 \Scal}{n} g(t), \qquad g(0)=g_0.
    \]
        Applying our computations above, we have the following system of equations for the normalized homogeneous Ricci flow on $\Spin(8)/\gG_2$:
        \begin{align*}
            &
            a'(t)=-\frac{3(a^4 - 4 a^3 b + 3 a^2 b^2 - 2 a^2 c^2 - 6 a b c^2 + 12 c^4)}{2 (a b - c^2)^2},\\
            &
            b'(t)=\frac{9 a b^3 - 12 b^2 ( a^2 + 3 c^2)  + 3 b (a^3 + 14 a c^2) - 6 c^2 (a^2 + 2 c^2)}{2 (a b - c^2)^2},\\
            &
            c'(t)=-\frac{3c (a^3 - 8 a^2 b + 3 a b^2 + 2 a c^2 + 6 b c^2)}{2 (a b - c^2)^2},\\
            & 
            a>0, \quad b>0, \quad ab-c^2>0.
        \end{align*}
    
    We will normalize the volume of $\Spin(8)/\gG_2$ such that 
    \begin{equation}\label{eq:volume_normalization}
        ab-c^2 = 1/2.
    \end{equation}
    Then normalized Ricci flow equation is given in the variables $a,b$ as follows:
    \begin{equation}\label{eq:sys_ab}
    \begin{aligned}
        &
        a'(t) = -18 - 6 a^4 + 36 a^3 b + (-54 b^2 - 6) a^2 + 54 a b,  \\
        &
        b'(t) = -6 a^3 b + (36 b^2 + 6) a^2 + (-54 b^3 - 18 b) a + 36 b^2 - 6,  \\
        &
        a>0, \quad b>0, \quad 2ab \geq 1.
    \end{aligned}
    \end{equation}
    We will call $\{(a,b)\in\bR^2:a>0,b>0,2ab\geq 1\}$ the set of admissible metrics.

    Fixed points for the flow occur at 
    \begin{equation}\label{eq:sing_ab}
        (a,b) = \of{ \tfrac{\sqrt{2}}{2} , \tfrac{\sqrt{2}}{2} }, \of{ \tfrac{\sqrt{6}}{4} , \tfrac{5\sqrt{6}}{12} }, \text{ and } \of{ \tfrac{\sqrt{6}}{2} , \tfrac{\sqrt{6}}{6} }.
    \end{equation}
    Each of these fixed points corresponds to a homogeneous Einstein metric on $\Spin(8)/\gG_2$ found by Kerr in \cite{Kerr98}.
    The metric induced by the Killing form on $\so(8)$ corresponds to $( \frac{\sqrt{2}}{2} , \frac{\sqrt{2}}{2} )$, while a product metric on $S^7 \times S^7 \cong \Spin(8)/\gG_2$ corresponds to $( \frac{\sqrt{6}}{4} , \frac{5\sqrt{6}}{12} )$.
    Finally, the metric corresponding to $(\frac{\sqrt{6}}{2} , \frac{\sqrt{6}}{6} )$ can be obtained from this product metric by conjugating by a rotation matrix:
    \[
        \begin{pmatrix}
            \frac{1}{2}         & -\frac{\sqrt{3}}{2} \\
            \frac{\sqrt{3}}{2}  & \frac{1}{2}
        \end{pmatrix}
        \begin{pmatrix}
            \frac{\sqrt{6}}{2}  & 0 \\
            0                   & \frac{\sqrt{6}}{6}
        \end{pmatrix}
        \begin{pmatrix}
            \frac{1}{2}         & \frac{\sqrt{3}}{2} \\
            -\frac{\sqrt{3}}{2} & \frac{1}{2}
        \end{pmatrix}
        =
        \begin{pmatrix}
            \frac{\sqrt{6}}{4}      & \frac{\sqrt{18}}{12} \\
            \frac{\sqrt{18}}{12}    & \frac{5\sqrt{6}}{12}
        \end{pmatrix}        
    \]
    As Kerr writes in \cite{Kerr98}, this rotation is the action of the triality automorphism of $\so(8)$, and the metric corresponding to $( \frac{\sqrt{6}}{2} , \frac{\sqrt{6}}{6})$ is isometric to a product metric.

    
        It is straightforward to check that the following are invariant curves for the flow of the System \eqref{eq:sys_ab}:
        \begin{enumerate}
            \item $2ab = 1$,                        
            \label{curve1ab}
            
            \item $a = 3b$,                         
            \label{curve2ab}
            
            \item $3( a^2 + b^2 ) + 2 = 10 a b$.    
            \label{curve3ab}
            
            \item $2a^2 + 3 = 6ab$.                 
            \label{curve4ab}
        \end{enumerate}
    (\ref*{curve1ab}) is the boundary curve for the set of admissible metrics, which corresponds to where $c=0$.
    We note that (\ref*{curve4ab}) can be obtained from (\ref*{curve2ab}) by applying the same transformation that sends $\of{ \frac{\sqrt{6}}{2} , \frac{\sqrt{6}}{6} }$ to $\of{ \frac{\sqrt{6}}{4} , \frac{5\sqrt{6}}{12} }$.
    
    Figure \ref{fig:phaseportait} illustrates a phase portrait for System \eqref{eq:sys_ab} on the left.
    The gray region consists of inadmissible values of $(a,b)$ for which $2ab<1$, the red points are fixed points for the system which correspond to Einstein metrics, the black dotted curves are the invariant curves we have listed above, and the green region is the set of metrics for which $\Ric>0$.


    \begin{figure}[H]
        \centering
        \includegraphics[width=0.49\textwidth]{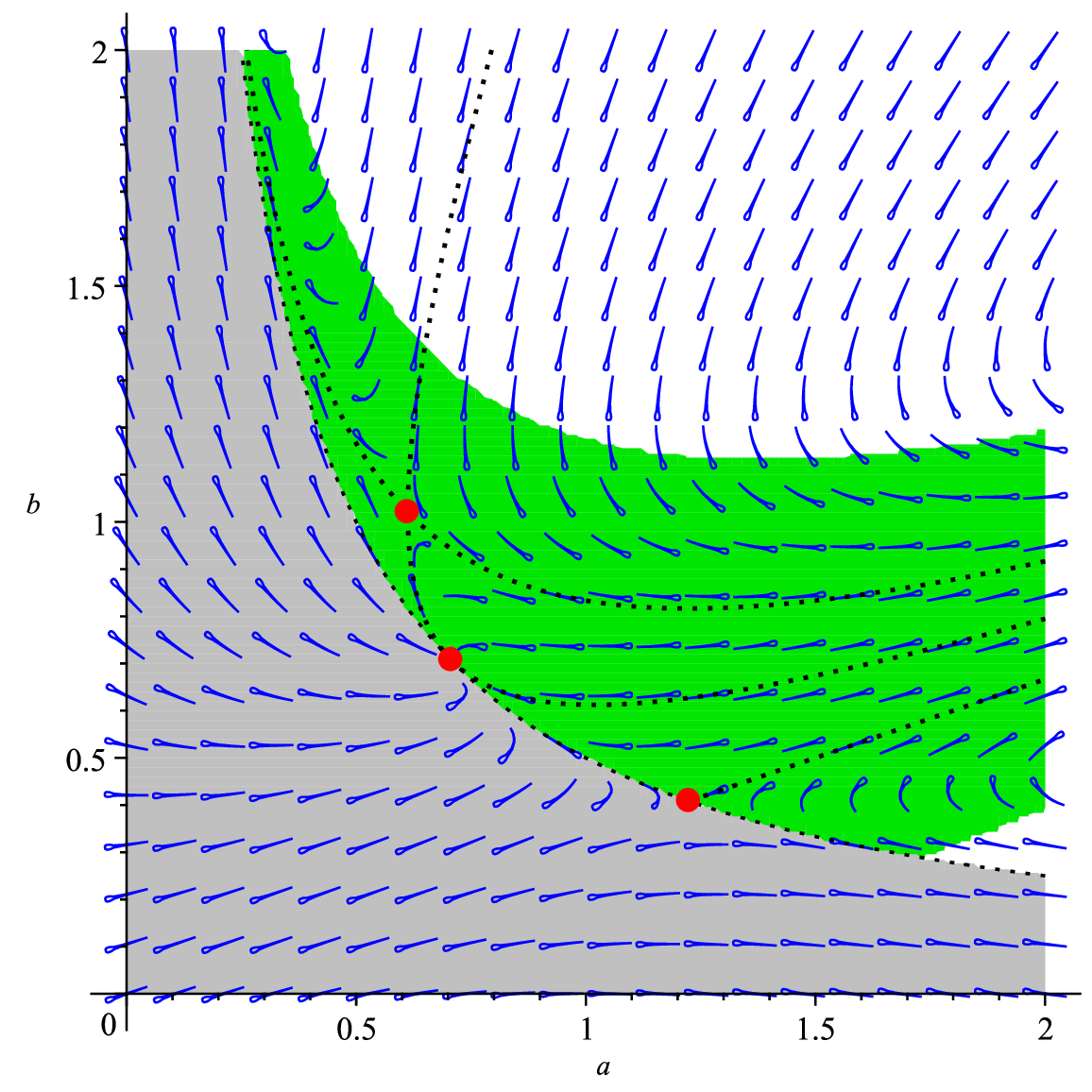}\includegraphics[width=0.49\textwidth]{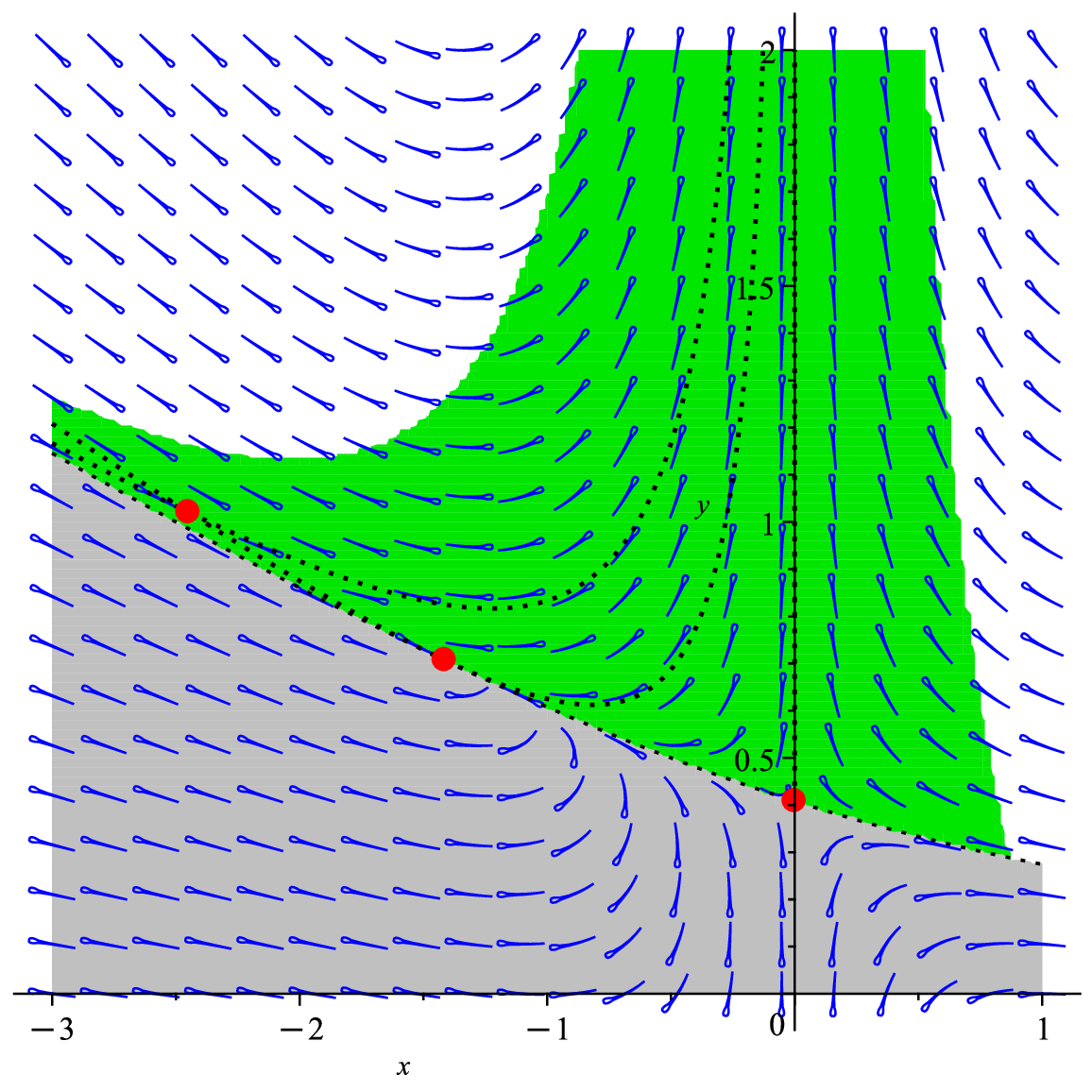}
        \caption{Phase portraits for Systems \eqref{eq:sys_ab} and \eqref{eq:sys_xy}.}
        \label{fig:phaseportait}
    \end{figure}

\section{Positive Ricci curvature in the long term}
\label{sec:mainproof}


    In this section, we will prove Theorem \ref{mainthm:Ricpos}.
    Namely, we will show that all solutions to the System \eqref{eq:sys_ab} eventually develop and maintain positive Ricci curvature.

    Because the line $a=3b$ is an invariant curve of the System \eqref{eq:sys_ab}, we find it convenient to implement the change of coordinates $x= a - 3b$ and $y=b$.  In these new coordinates, System \eqref{eq:sys_ab} becomes:
    \begin{equation}\label{eq:sys_xy}
    \begin{aligned} 
        &
        x'(t) = -6x\left(x^3 + 3x^2 y + 4x + 6y\right), \\
        &
        y'(t) = -6x^3 y - 18x^2 y^2 + 6x^2 + 18x y + 36y^2 - 6. \\
        &
        2xy+6y^2\geq 1, \quad y>0  .
    \end{aligned}
    \end{equation}
    Figure \ref{fig:phaseportait} illustrates a phase portrait for System \eqref{eq:sys_xy} on the right. 
    The fixed points from \eqref{eq:sing_ab} are given by
    \begin{equation}\label{eq:sing_xy}
        (x,y) = \of{-\sqrt{2},\tfrac{\sqrt{2}}{2}}, \of{-\sqrt{6},\tfrac{5\sqrt{6}}{12}}, \text{ and } \of{0,\tfrac{\sqrt{6}}{6}}.
    \end{equation}
    The invariant curves from \eqref{curve1ab}--\eqref{curve4ab} above are given in the $(x,y)$-coordinates by
    \begin{enumerate}
        \item \label{curve1xy} $2xy + 6y^2 = 1$ (with $y>0$),
        \item \label{curve2xy} $x=0$,
        \item \label{curve3xy} $3x^2 + 8xy + 2 = 0$ (with $x<0$), and 
        \item \label{curve4xy} $2x^2 + 6xy + 3 = 0$ (with $x<0$). 
    \end{enumerate}
    The entries of the Ricci $(1,1)$-tensor from \eqref{eq:Ric11abc} become
    \begin{align*}
        \delta &= 2 x^3 + 18 x^2 y + 36 x y^2 + 18 y, \\
        \epsilon &= 3 (x^2 + 2 x y - 1) \sqrt{4 x y + 12 y^2 - 2} , \\
        \zeta &= -3 (2x^2 + 6 x y + 3) \sqrt{4 x y + 12 y^2 - 2} ,\\
        \eta &= -2 (2x^2 + 6 x y - 3) (x + 3 y).
    \end{align*}
    Thus, the sum $S$ and product $P$ of the eigenvalues of the Ricci tensor (counted without multiplicity) are 
    \begin{equation}\label{eq:Pxy}
    \begin{aligned}
        S &= -2 x^3 - 6 x^2 y + 6 x + 36 y,  \\
        P &= 54 - 8 x^6 - 48 x^5 y + (-72 y^2 - 24) x^4 - 72 x^3 y - 18 x^2.
    \end{aligned}
    \end{equation}

    Throughout this section, $(x(t),y(t))$ will refer to a solution to the System \eqref{eq:sys_xy}, and $g(t) = g(x(t),y(t))$ will denote the associated metrics.
    When discussing aspects of these objects that are not time dependent, we will refer to them simply as $(x,y)$ and $g$, respectively.

    We first show that it is enough to check for the positivity of $P$:

    
    \begin{lemma}\label{lem:ISTS_P>0}
        A metric $g$ has $\Ric > 0$ if and only if $P > 0$.
    \end{lemma}

    \begin{proof}


        Because $P$ is the product of eigenvalues of the  Ricci tensor, it is clear that $\Ric> 0$ implies $P>0$. 
        Conversely, using Maple's RegularChains package and SemiAlgebraicSetTools subpackage, the IsEmpty function shows that $\{(x,y) : y>0, \ 2 x y + 6 y^2 \geq 1, \ P > 0, \ S \leq 0\}$ is empty.
        Thus if $P>0$, then $S>0$, and because the Ricci tensor has only two eigenvalues (counted without multiplicity), it follows that $\Ric>0$.
        %
        %
        %
    \end{proof}

    \begin{remark}\label{rmk 1}
        Setting $P$ from \eqref{eq:Pxy} equal to zero and solving for $y$ yields:
        \[
        y_1(x) = -\frac{x}{3} - \frac{1}{2x} - \frac{\sqrt{3}}{2x^2},
        \qquad
        y_2(x) = -\frac{x}{3} - \frac{1}{2x} + \frac{\sqrt{3}}{2x^2}.
        \]
        Notice $y_1<y_2$ for all $x\neq0$.
        Since the coefficient of $y^2$ is negative in \eqref{eq:Pxy}, we have that $P$ is positive only when $y_1<y<y_2$. 
        Furthermore, if $y>0$ and $2xy+6y^2\geq 1$, then using Maple's IsEmpty function, one can show $y_1 < y$.
        Therefore, we have $P>0$ if and only if $y<y_2$.
    \end{remark}

    Next, we show that when composed with a solution of system \eqref{eq:sys_xy}, $P$ is a strictly increasing function of $t$ almost everywhere.
    First, denoting $P|_{(x(t),y(t))}$ as simply $P(t)$, we remark that
    \begin{align*}
        P'(t) &= 72x^2\left(2x^2 + 6xy + 3\right)^2 \left(x^3 + 3x^2y + 2x + 2y \right).
    \end{align*}
    
    \begin{lemma}
        \label{lem:Pinc}
        For any $(x,y)$ lying in the region where $2xy+6y^2>1$ and $y>0$, we have $P'(t)>0$ for all $t$ except on the two invariant curves $x=0$ and $2x^2+6xy+3=0$ (with $x<0$), on which $P'(t)$ is zero. 
        Additionally, these two curves are contained in the region where $P>0$.
        Consequently, the region where $P>0$ is forward-invariant.
    \end{lemma}
    
    \begin{proof}
        First, it is again verifiable using Maple's IsEmpty function that $P'(t)\geq 0$ when $2xy+6y^2 > 1$ and $y>0$. 
        Because $x$ and $2x^2+6xy+3$ are factors of $P'(t)$, the curves $x=0$ and $2x^2+6xy+3=0$ satisfy $P'(t)=0$.
        The zeros of the remaining factor $x^3 + 3x^2y + 2x + 2y$ of $P'(t)$ lie outside of the set where $2xy+6y^2 > 1$, meaning the invariant curves $x=0$ and $2x^2+6xy+3=0$ constitute the only points where $P'(t)=0$.
        By uniqueness of solutions, it follows that any solution whose initial condition satisfies $2xy+6y^2 > 1$ and lies outside of these curves must have $P'(t)>0$ for all $t$. 
        The claim that all points on these invariant curves satisfy $P>0$ is also verifiable with Maple.
    \end{proof}

    The fixed points of System \eqref{eq:sys_xy} given in \eqref{eq:sing_xy} are all contained in the region where $P>0$, and by Lemma \ref{lem:Pinc}, so are the invariant curves \eqref{curve2xy} and \eqref{curve4xy}.
    The other invariant curves \eqref{curve1xy} and \eqref{curve3xy} with the fixed points deleted consist of components that either are contained in the region where $P>0$ or serve as a component of the stable manifold for one of the saddle point singularities at $(x,y)=(-\sqrt{6},\frac{5\sqrt{6}}{12})$ and $(0,\frac{\sqrt{6}}{6})$.
    In particular, we have verified that the conclusion of Theorem \ref{mainthm:Ricpos} holds for initial conditions lying on any of the curves \eqref{curve1xy}--\eqref{curve4xy}.
    
    Due to this fact, below we only work with initial conditions which lie in the complement to the union of these invariant curves.
    Within the interior of the space of admissible metrics, $M=\{(x,y): y > 0 \text{ and } 2xy+6y^2 > 1\}$, this complement is partitioned into the following connected components:
    \begin{equation}\label{regions}
    \begin{aligned}
        A &= \{(x,y)\in M : 3x^{2}+8xy+2 > 0 \text{ and } 2x^2 + 6xy + 3 < 0 \},  \\
        B &= \{(x,y)\in M : x < -\sqrt{2}, 3x^{2}+8xy+2 > 0, \text{ and } 2x^2 + 6xy + 3 > 0 \},   \\
        C &= \{(x,y)\in M : 3x^{2}+8xy+2 < 0  \text{ and } 2x^2 + 6xy + 3 < 0 \},   \\
        D &= \{(x,y)\in M : 3x^{2}+8xy+2 < 0 \text{ and } 2x^2 + 6xy + 3 > 0 \},   \\
        E &= \{(x,y)\in M : -\sqrt{2} < x < 0 \text{ and } 3x^{2}+8xy+2 > 0 \},   \\
        F &= \{(x,y)\in M : x > 0 \}.   
    \end{aligned}
    \end{equation}
    These components are illustrated in Figure \ref{fig:regions}.
    Because they are bounded by invariant curves, they are forward-invariant for the flow of \eqref{eq:sys_xy}.
    Thus, it suffices to prove Theorem \ref{mainthm:Ricpos} for solutions whose initial conditions lie in each of these regions.

    \begin{figure}[H]
        \centering
        \includegraphics[width=0.5\linewidth]{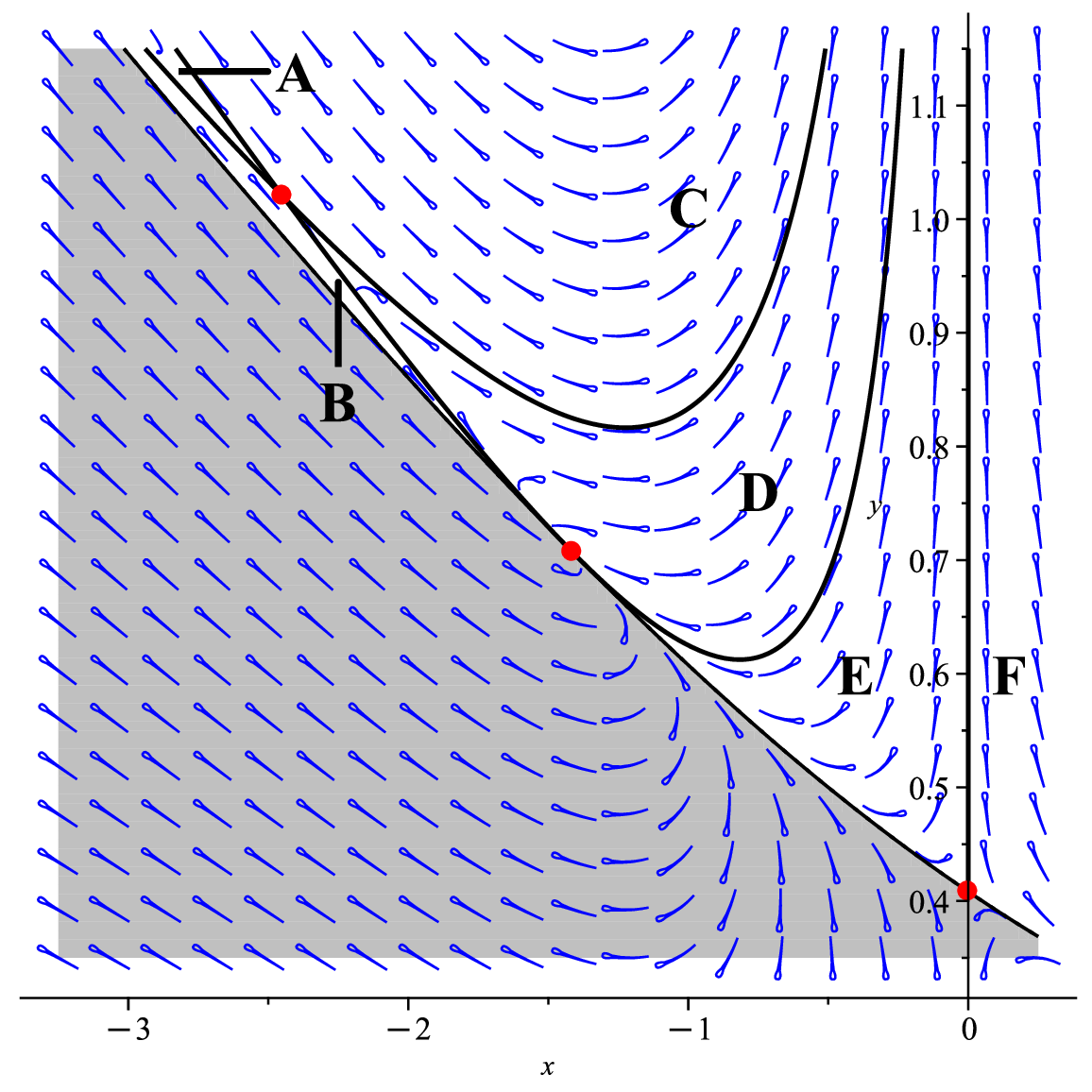}
        \caption{Regions $A$--$F$ defined in \eqref{regions}.}
        \label{fig:regions}
    \end{figure}

    \begin{remark}\label{rem:compact}
        Note that by Lemma \ref{lem:Pinc}, System \eqref{eq:sys_xy} has no cyclic solutions, as such a solution would necessarily have $P'(t) \equiv 0$, and this equality only holds on the unbounded curves \eqref{curve2xy} and \eqref{curve4xy}.
        Consequently by the Poincar\'e-Bendixson theorem, because the fixed points of the system consist of a source at $(x,y)=(-\sqrt{2},\frac{\sqrt{2}}{2})$ and saddle points at $(x,y)=(-\sqrt{6},\frac{5\sqrt{6}}{12})$ and $(0,\frac{\sqrt{6}}{6})$, the only solutions which are contained in a compact region are the fixed points and solutions on the curves which serve as the stable manifolds for the saddle points.
        As stated above, these stable manifolds lie on curves \eqref{curve1xy} or \eqref{curve3xy}.
        Therefore, because we only need to work with initial conditions which do not lie on these invariant curves, we may assume all solutions analyzed below are unbounded.
    \end{remark}

    It can be verified with Maple's IsEmpty function that the regions $B$, $D$, and $E$ are contained in the region where $P>0$.
    Thus, we have the following:
    
    \begin{proposition}
        If $(x(t),y(t))$ is a solution of System \eqref{eq:sys_xy} whose initial condition lies in region $B$, $D$, or $E$ from \eqref{regions}, then the solution has $\Ric>0$ for all $t$.
    \end{proposition}

    Our next goal is to prove Theorem \ref{mainthm:Ricpos} for the remaining solutions that approach the line $x=0$, i.e. those contained in $C$ and $F$.
    We begin with the following:

    \begin{lemma}\label{lemma D}
        Suppose $(x(t),y(t))$ is a solution of System \eqref{eq:sys_xy} whose initial condition lies in region $C$ or $F$ from \eqref{regions}.
        Then $|x(t)|$ is monotonically decreasing for all $t$ and $x(t)\to 0$ and $t$ increases, and $y(t)$ is monotonically increasing for sufficiently large values of $t$ and $y(t) \to \infty$ as $t$ increases.
    \end{lemma}


    \begin{proof}
        If the initial condition lies in region $F$ (where $x>0$), then $x(t)>0$ for all $t$, and thus $x'(t) = -6x\left(x^3 + 3x^2 y + 4x + 6y\right)<0$ for all $t$.
        If instead the initial condition lies in region $C$, then it can be verified with Maple's IsEmpty function that $x'(t)>0$ for all $t$.
        Therefore, $|x(t)|$ is monotonically decreasing for all $t$ in both cases.
        
        Next, we show that $y(t)$ eventually increases monotonically.
        By Remark \ref{rem:compact}, we may assume that our solution $(x(t),y(t))$ is unbounded.
        Since $|x|$ is bounded, we have that $|y|$ is unbounded.  
        Because $y>0$ by assumption, the solution must satisfy $y'(t) > 0$ for some $t$.  
        Using Maple's IsEmpty function, we can verify if $y>2$, then $y''(t)>0$.
        Thus the region where $y>2$ and $y'(t)>0$ is forward-invariant, and because $y$ is unbounded, our solution must eventually enter and remain in this region.
        Hence, $y(t)$ increases monotonically for all sufficiently large values of $t$, and because $y$ is unbounded, we have $y\to \infty$ as $t$ increases.



        Consequently, we may view $y$ as an invertible function of $x$ for sufficiently large values of $t$.
        Thus, considering $\frac{dy}{dx} = \frac{y'(t)}{x'(t)}$, we have by \eqref{eq:sys_xy} that 
        \[ 
            \frac{dy}{dx} = \frac{-6x^3 y - 18x^2 y^2 + 6x^2 + 18x y + 36y^2 - 6}{-6x\left(x^3 + 3x^2 y + 4x + 6y\right)}. 
        \]
        Using that $|x|$ is bounded, we get that
        \begin{equation}\label{eqn A}
             \lim_{y \to \infty}\frac{dy}{dx} \cdot \frac{1}{y} = \lim_{y \to \infty}\frac{x^2-2}{x\left(x^2+2\right)}. 
        \end{equation}
        Because $|x(t)|$ is bounded and monotonic, it must converge to some finite value as $t$ increases (equiv. $y\to \infty$).
        For the sake of contradiction, assume this limit is non-zero.
        Then the limit \eqref{eqn A} converges to some finite value, and thus there exists a positive number $M$ such that, for sufficiently large values of $t$, we have
        \[
            \left | \frac{dy}{dx} \cdot \frac{1}{y} \right | < M.
        \]
        Because $|x(t)|$ is bounded, the above inequality implies that $|y(t)|$ must also be bounded, which is a contradiction.
        Thus, we have $x\to 0$ as $t$ increases.
    \end{proof}
    

    We are now ready to prove the following case of Theorem \ref{mainthm:Ricpos}:

    \begin{theorem}\label{thm:P>0firstcase}
    Suppose $(x(t),y(t))$ is a solution of System \eqref{eq:sys_xy} whose initial condition lies in region $C$ or $F$ from \eqref{regions}.
    Then the solution will satisfy $\Ric>0$ for all sufficiently large values of $t$.
\end{theorem}

\begin{proof}
    
    
    By Lemma \ref{lemma D}, $|x(t)|$ converges monotonically to $0$ and $y(t)$ is unbounded and eventually grows monotonically as $t$ increases.
Thus, we may again treat $y$ as a function of $x$, and we may consider
    \begin{align*}
        \frac{dy}{dx}\cdot \frac{x}{y} 
        = \frac{y'(t)}{x'(t)}\cdot \frac{x}{y} 
        &= \frac{ (3x^2-6)y^2 + (x^3 - 3x)y + 1 - x^2}{ (3x^2+6)y^2 + (x^3  + 4x)y }.
    \end{align*}
    Because $y\to\infty$ as $x\to 0$ (equiv. as $t$ increases), we get that
    \[
        \lim_{x \to 0} \left|\frac{dy}{dx}\cdot \frac{x}{y}\right| \leq  1.
    \]
    Hence given $\epsilon>0$, for small enough values of $x$, we have
    \[
        \left|\frac{dy}{dx}\right| < (1+\epsilon)\frac{y}{|x|}
    \]
    Now, solutions of the differential equation $\frac{dy}{dx} = -(1+\epsilon)\frac{y}{x}$ are given by $y = A|x|^{-(1+\epsilon)}$ for some constant $A$.
    Thus, for sufficiently large values of $t$, the solution $(x(t),y(t))$ satisfies $y(x) < A|x|^{-(1+\epsilon)}$.
    In particular, choosing $\epsilon <1$, we have $A|x|^{-(1+\epsilon)}<y_2(x)$ for sufficiently small values of $x$, where $y_2$ is the curve described in Remark \ref{rmk 1}.
    Thus, for large enough values of $t$, we get that $(x(t),y(t))$ satisfies $y<y_3<y_2$, and hence $P>0$, which implies $\Ric>0$.
%
%
\end{proof}


It remains to show Theorem \ref{mainthm:Ricpos} holds in region $A$.
First, we show the following:

\begin{lemma}\label{lemma E}
    Suppose $(x(t),y(t))$ is a solution of System \eqref{eq:sys_xy} whose initial condition lies in region $A$ from \eqref{regions}.
    Then $x(t)$ decreases monotonically for all sufficiently large values of $t$, and $x(t) \to -\infty$ as $t$ increases.
\end{lemma}

\begin{proof}
    %
    By Remark \ref{rem:compact}, we know that our solution is unbounded. 
    Within the region $A$, because it is bounded by the invariant curves $3x^{2}+8xy+2 = 0$ and $2x^2 + 6xy + 3 = 0$ (with $x<-\sqrt{6}$), a solution contained in $A$ is unbounded only if $x(t)\to -\infty$ and $y(t) \to \infty$ as $t$ increases.
    In particular, there exists $t$ for which $x'(t) < 0$ and $y'(t) > 0$.
    Furthermore, it is verifiable with Maple's IsEmpty function that in region $A$, if $y'(t)\geq 0$, then $x'(t)<0$ and $x''(t)<0$.
    Thus the subset of $A$ for which $x'(t)<0$ is forward-invariant.
\end{proof}



    Finally, we prove the last case of Theorem \ref{mainthm:Ricpos}:
    
    \begin{theorem}\label{thm:enter-P-positive-xy}
        Suppose $(x(t),y(t))$ is a solution of System \eqref{eq:sys_xy} whose initial condition lies in region $A$ from \eqref{regions}.
        Then the solution will satisfy $\Ric>0$ for all sufficiently large values of $t$.
    \end{theorem}

    \begin{proof}
        By Lemma \ref{lemma E}, $x(t)$ eventually decreases monotonically, so we may regard $y(t)$ and $P(t)$ as functions of $x$.
        From \eqref{eq:sys_xy} and \eqref{eq:Pxy}, we get that
        \begin{equation}\label{eq:dP_dx_exact}
            \begin{aligned}
                \frac{dP}{dx}&=
                -\frac{72x^{2}\big(6xy+2x^{2}+3\big)^{2}\big(x^{3}+3x^{2}y+2y+2x\big)}
                {6x\big(x^{3}+3x^{2}y+4x+6y\big)}\\
                &= -12x\,(2x^2+6xy+3)^2
                \left(1-\frac{2x+4y}{x^3+3x^2y+4x+6y}\right).
            \end{aligned}
        \end{equation}
        Suppose the solution initially satisfies $P\leq 0$.
        Defining 
        \[ 
            P_1 = 2x^2+6xy+3 \quad \text{ and } \quad P_2=1-\frac{2x+4y}{x^3+3x^2y+4x+6y},
        \]
        we will first estimate $P_1$ and $P_2$ in order to bound $\frac{dP}{dx}=-12x(P_1)^2 P_2$ above by a negative function of $x$, which corresponds to a positive lower bound on $\frac{dP}{dt}$ since $x(t)$ decreases monotonically.

        \begin{figure}[H]
            \centering
            \includegraphics[width=0.45\textwidth]{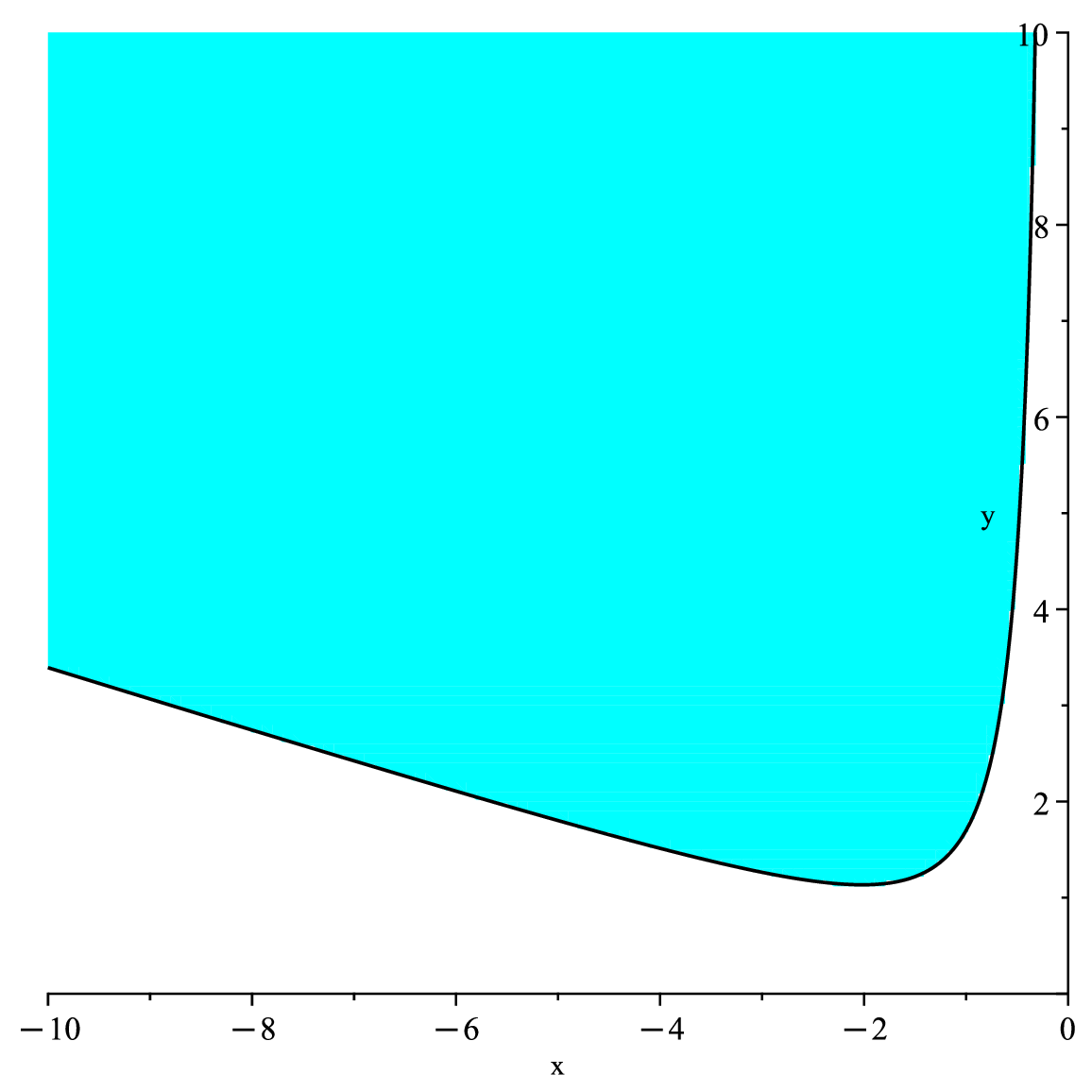}
            \caption{The region where $P\leq0$ is bounded by $y=y_2(x)$.}
            \label{fig:negativep}
        \end{figure}
        
        First notice that in the region where $P\leq 0$, for a fixed value of $x$, the minimal value of $y$ occurs on the curve where $P=0$, which is given by $y_2(x) = -\frac{x}{3} - \frac{1}{2x} + \frac{\sqrt{3}}{2x^2}$ in Remark \ref{rmk 1}; see Figure \ref{fig:negativep}.
        We also have that $\frac{\partial P_1}{\partial y}=6x<0$ in the region $A$.
        Thus for a fixed value of $x$, the maximal value of $P_1$ is attained on $y_2(x)$.
        Hence for all $x$,
        \begin{equation}\label{eq:P1}
            P_1 = 2x^2 + 6xy + 3 \leq \frac{3\sqrt{3}}{x} < 0.
        \end{equation}

        Next notice, because $x<-\sqrt{6}$ in region $A$, we have
        \[
            \frac{\partial P_2}{\partial y} = \frac{2x(x^2-2)}{(x^3+3x^2y+4x+6y)^2} < 0.
        \]
        So similarly, for a fixed value of $x$, the maximal value of $P_2$ is attained on $y_2(x)$.
        Hence for all $x$,
        \[
            P_2 = 1-\frac{2x+4y}{x^3+3x^2y+4x+6y} \leq - \frac{x^3-9\sqrt{3}x^2+6x-6\sqrt{3}}{3x^3+9\sqrt{3}x^2-18x+18\sqrt{3}}.
        \]
        Consequently, 
        \begin{equation}\label{eq:P2}
            \lim_{x\to-\infty} P_2 \leq -\frac{1}{3}.
        \end{equation}

        Combining \eqref{eq:P1} and \eqref{eq:P2}, for any sufficiently small $\epsilon>0$ and for sufficiently large values of $x$, we have 
        \[
            \frac{dP}{dx} \leq -12x \of{\tfrac{3\sqrt{3}}{x}}^2 \of{-\tfrac{1}{3}+\epsilon} = \frac{324 (1-3\epsilon)}{3x} < 0.
        \]
        Consequently, there exist constants $C_1, C_2$ such that $C_1>0$ and for sufficiently large values of $x$,
        \[P > C_1\ln |x|+C_2.\]
        Because $x\to -\infty$ as $t$ increases, we conclude that $P > 0$ for sufficiently large values of $t$.
    \end{proof}

\bibliographystyle{abbrv}
\bibliography{biblio}

\end{document}